\documentclass[12pt,reqno]{amsart}

\usepackage{amssymb,amsfonts,latexsym,amstext,epsfig,color}

\usepackage[left=3cm,top=3cm,right=3cm,bottom=3cm]{geometry}

\newtheorem{theorem}{Theorem}
\newtheorem{lemma}[theorem]{Lemma}
\newtheorem{corollary}[theorem]{Corollary}

\begin{document}

\title{Hyperopic Cops and Robbers}
\author{A.\ Bonato}
\author{N.E.\ Clarke}
\author{D.\ Cox}
\author{S.\ Finbow}
\author{F.\ Mc Inerney}
\author{M.E.\ Messinger}

\thanks{The authors gratefully acknowledge support from NSERC. The fifth author is partially supported by ANR project Stint under reference ANR-13-BS02-000,
the ANR program ``Investments for the Future'’ under reference ANR-11-LABX-0031-01, and the Inria Associated Team AlDyNet.}

\begin{abstract}
We introduce a new variant of the game of Cops and Robbers played on graphs, where the robber is invisible unless outside the neighbor set of a cop. The hyperopic cop number is the corresponding
analogue of the cop number, and we investigate bounds and other properties of this parameter. We characterize the cop-win graphs for this variant, along with graphs with the largest possible
hyperopic cop number. We analyze the cases of graphs with diameter 2 or at least 3, focusing on when the hyperopic cop number is at most one greater than the cop number. We show that for planar
graphs, as with the usual cop number, the hyperopic cop number is at most 3. The hyperopic cop number is considered for countable graphs, and it is shown that for connected chains of graphs, the
hyperopic cop density can be any real number in $[0,1/2].$
\end{abstract}

\maketitle

\section{Introduction}

In the game of Cops and Robbers, the robber is visible throughout the game. Perfect information, however, may be less realistic; it is possible that only certain moves make the robber visible.
Several pursuit and evasion games have been studied with imperfect information and this is a common theme in graph searching; see \cite{by,fomin} for surveys. A recent variant \cite{vis} of Cops and
Robbers fixes a visibility threshold for the cops, where the cops can only see vertices within their $k$th neighborhood, for a fixed non-negative integer $k$.

We consider a variant of Cops and Robbers where the cops can only see vertices \emph{not} in their neighbor set. Hence, the robber remains invisible when they are ``close enough''; that is, within
distance one of all the cops. A motivation for this comes from considering adversarial networks, where edges denote competition, enmity, or rivalry between vertices. For example, the market graph is
studied in \cite{market}, where stocks are adjacent if they are negatively correlated. Other examples of adversarial networks are discussed in \cite{market,ILAT,guo}. The robber, as a rogue agent,
cannot be seen while moving among adversarial vertices.  We refer to this as \emph{Hyperopic Cops and Robbers}, as hyperopia is the condition of farsightedness.

The game of Hyperopic Cops and Robbers is defined on (reflexive) graphs more precisely as follows. There are two players, with one player controlling a set of \emph{cops}, and the second controlling
a single \emph{robber}. We use inclusive, gender neutral pronouns ``they'' and ``them'' when referring to the players (and specify the players as cops or robber if they are less clear from context).
Unlike in Cops and Robbers, the cops play with imperfect information: the robber may be invisible to the cops during gameplay. The robber, however, sees the cops at all times. The game is played over
a sequence of discrete time-steps; a \emph{round} of the game is a move by the cops together with the subsequent move by the robber. The cops and robber occupy vertices, and when players are ready to
move during a round, they may each move to a neighboring vertex. The cops move first, followed by the robber; thereafter, the players move on alternate steps. Players can \emph{pass}, or remain on
their own vertices by moving on loops. Observe that any subset of cops may move in a given round. The robber is invisible if and only if they are adjacent to all the cops. If the robber occupies the
same vertex as a cop, then they are visible.

The cops win if, after some finite number of rounds, a cop occupies the same vertex as the robber. This is called a \emph{capture}. The robber wins if they can evade capture indefinitely. Further, we
insist that there be \emph{no chance} in capturing the robber; that is, in the final move, the cops' moves guarantee a cop occupies the vertex of the robber. To illustrate this point, consider one
cop playing on a triangle $K_3$, first occupying a vertex $u.$ The robber occupies a vertex distinct from $u$ and is invisible; further, the cop does not know which vertex to move for capture. Hence,
one hyperopic cop is insufficient to capture the robber on $K_3.$

Note that if a cop is placed at each vertex, then the cops see the robber in the initial round and capture them in the next round. Therefore, the minimum number of cops required to win in a graph $G$
is a well-defined positive integer, named the \emph{hyperopic cop number} of the graph $G$.  The notation $c_H(G)$ is used for the hyperopic cop number of a graph $G$. If $c_H(G)=k,$ then $G$ is
$k$-\emph{hyperopic cop-win}. In the special case $k=1,$ $G$ is \emph{hyperopic cop-win}.

For a graph $G$, it is evident that $c(G) \le c_H(G)$, where $c(G)$ denotes the cop number in the traditional game of Cops and Robbers. Further, for a disconnected graph $G$, there must be a cop in each
component; hence, the robber is always visible, and so we have that $c_H(G)=c(G)$. For that reason, we only consider connected graphs.

While a priori it might appear that there should be a relationship between the hyperopic cop number and the $1$-visibility cop number $c_{v,1}$ from \cite{vis} (as hyperopic cops only see the robber
on the complements of their neighbor sets), there is no elementary relationship between the parameters. In particular, there are graphs $G$ such that $c_H(G) \neq c_{v,1}(\overline{G}),$ where
$\overline{G}$ is the complement of $G$. To see this, note that $c_H(K_n)=\left\lceil \frac n2 \right\rceil$ (which we will prove in Theorem~\ref{tmax}), while $c_{v,1}(\overline {K_n})=n$. Hence,
there exist families of graphs $G$ such that $|c_H(G)-c_{v,1}(\overline {G})|$ is arbitrarily large. An example of a connected graph where its complement is also connected and these two parameters
differ is $K_{m,n}-e$ for $m,n\geq 3$, where $e$ is any fixed edge. In this case, $c_H(K_{m,n}-e)=2$ and $c_{v,1}(\overline {K_{m,n}-e})=1$.

The paper is organized as follows. In Section~\ref{seccw}, we first characterize the hyperopic cop-win graphs and then give general upper bounds depending on properties of the graph. We consider the
cases for diameter at least 3 graphs and diameter 2 graphs in the next section, focusing on results where $c_H$ is at most one larger than the cop number. In Section~\ref{secplan}, we prove that the
hyperopic cop number is at most three for planar graphs and two for outerplanar graphs, paralleling analogous results for the cop number. We finish with a discussion of hyperopic densities of
infinite graphs. Our main result in Section~\ref{secdense} is that hyperopic densities with connected chains may be any real number in $[0,1/2]$, which is in stark contrast to the cop densities of
connected chains which are always 0.

All graphs we consider are connected and undirected, unless otherwise stated. As referenced above, our graphs are reflexive, but we do not allow multiple edges. All our graphs are finite, except in
Section~\ref{secdense}. For background on graph theory, see \cite{west}.

\section{Hyperopic cop-win graphs and bounds}\label{seccw}
We begin by characterizing the graphs with hyperopic cop number equaling 1.

\begin{theorem}\label{ttree}
For a graph $G$, $c_H(G)=1$ if and only if $G$ is a tree.
\end{theorem}

\begin{proof}
Let $G$ be a tree. The cop begins on a end-vertex $u$. If the robber is invisible, they must be located on the unique neighbor $v$ of $u$.  The cop wins by moving to $v$. Hence, in the initial round
we may assume the robber is visible.

In all subsequent rounds, if the robber becomes invisible, then the cop sees which of their neighbors the robber moves to and captures them. Otherwise, the robber remains visible (when either passing
or moving). As the rounds progress, the distance between the cop and robber monotonically decreases, and eventually the robber occupies an end-vertex whose unique neighbor is occupied by the cop. The
robber becomes invisible in this stage of the game, but the cop knows their location and captures them.

Now suppose that $G$ contains a cycle $H$. We show that however the cop moves, the robber can survive the round by either moving or passing on the cycle. The proof then follows by induction.

If the cop is not adjacent to the robber, then the robber passes.   If the cop is adjacent to the robber and it is the robber's turn to move, then either there is a neighbor of the robber on $H$ that
is not adjacent to the cop (and the robber moves there and is safe for another round), or the cop is adjacent to all the neighbors of the robber on $H$. In this case, the cop cannot guarantee capture
as the robber can occupy one of two or more possible positions on the next move. The robber wins as no chance is allowed when the cop captures the robber.
\end{proof}

In the next result, we bound $c_H$ by a value close to the cop number.

\begin{theorem}\label{ttf}
Let $G$ be a graph.
\begin{enumerate}
\item If $G$ contains a cut vertex, then $c_H(G)\leq c(G)+1$.
\item If $G$ is triangle-free, then $c_H(G)\leq c(G)+1$.
\end{enumerate}
\end{theorem}

\begin{proof} For (1), by Theorem~\ref{ttree}, we need only to consider $G$ with $c_H(G) \ge 2$. Let $u$ be a cut vertex, and let $v$ and $w$ be vertices adjacent to $u$ that are in
different components of $G-u.$

Place one cop on $v$ and the remaining cops on $w.$ If the robber is invisible, then they must be on $u$, in which case the cops can capture the robber in the next round.

Therefore, the robber will choose to start off of $u$, and so is visible in the initial round. Suppose that $R$ is in the same component as $v$ (the case when $R$ is in the component of $w$ is
analogous). One cop remains on $w$ throughout the game and ensures that the robber remains visible unless the robber moves to $u$ (in which case they are captured). Hence, the robber remains visible
for the remainder of the game, and the remaining $c(G)$-many cops can play their winning strategy in $G$ and capture the robber.

For (2), let $c(G)=m$ and consider $m$-many cops $C_1, C_2, \ldots, C_m$ playing a winning strategy $\mathcal{S}$ in $G$ in the usual game of Cops and Robbers. Add one cop $C'$ who remains adjacent
(but not equal) to a fixed cop $C_1$ at all times. The robber cannot be adjacent to both $C_1$ and $C'$ as there are no triangles, and so remains visible throughout the game. Hence, $m+1$ cops win
playing Hyperopic Cops and Robbers on $G$.
\end{proof}

We next consider how large $c_H$ can be for a connected graph.

\begin{theorem}\label{tmax}
For a graph $G$ on $n$ vertices, $c_H(G)\leq \lceil \frac{n}{2} \rceil$. The bound is tight as witnessed by cliques.
\end{theorem}

\begin{proof}
Let $m=\lceil \frac{n}{2} \rceil$. It is straightforward to see that $\gamma(G) \le \frac{n}{2}$ for a connected graph $G$ of order $n$, where $\gamma(G)$ is the domination number of $G$. The
$m$-many cops start the game by occupying a dominating set $S$ of cardinality $m$. If the robber is visible, then the cops win. If the robber is invisible, then they must be adjacent to all of the
cops. In this case, there are at most $\lfloor \frac{n}{2} \rfloor$ vertices that may be occupied by the robber, and each of these vertices must be adjacent to all the cops' positions. The cops then
move to all these positions to capture the robber, which is possible since $\lfloor \frac{n}{2} \rfloor \leq m$.

In a clique $K_n$, if fewer than $m$-vertices are occupied by cops, then they cannot infer the robber's position. Hence, the robber wins. \end{proof}

Note that cliques are not the only graphs witnessing the $\lceil \frac{n}{2} \rceil$ upper bound.

\begin{theorem} \label{tke}
For all $n\ge 3,$ we have that $c_H(K_n-e)=\lfloor \frac{n}{2} \rfloor,$ where $K_n-e$ is $K_n$ with one edge removed.
\end{theorem}

\begin{proof}
As the case for $n=3$ is straightforward, we consider $n \ge 4.$ Suppose $e=uv$ and let $X=V(K_n-e) \setminus \{u,v\}$. For the lower bound, assume there are at most $\lfloor \frac{n}{2}
\rfloor-1$-many cops. The robber's strategy is to remain in or move to $X$ when either $u$ or $v$ is occupied by a cop and this is possible as $|X|\ge 2$. Otherwise, the robber does not restrict
themselves to any subgraph.

If there is a cop on $u$ or $v$, then after the robber's move, the robber is in $X$ and therefore, is invisible and the cops know they are not on either $u$ or $v$. In this case, the cops can only
occupy at most $\lfloor \frac{n}{2} \rfloor-2$ of the vertices of $X$. By moving, the cops can only cover another $\lfloor \frac{n}{2} \rfloor-1$ of the vertices of $X$. As $\lfloor \frac{n}{2}
\rfloor-2+\lfloor \frac{n}{2} \rfloor-1=2\lfloor \frac{n}{2} \rfloor-3< n-2$, the robber may evade capture.

If there is no cop on $u$ or $v$, then the robber is invisible and may be on any vertex of $K_n-e$. The cops could then occupy and move to at most $2\lfloor \frac{n}{2} \rfloor-2<n$ vertices. Therefore, the robber
evades capture in this situation, and the lower bound follows.

To prove $c_H(K_n-e)\leq \lfloor \frac{n}{2} \rfloor$, we consider the following cop strategy. Note that we have at least two cops for any of the graphs we are considering. The cops occupy $u$ and
$\lfloor \frac{n}{2} \rfloor-1$ other distinct vertices. The cops can see if the robber occupies $v$ and so the robber never moves there. Thus, the cops move to the remaining $(n-2)-(\lfloor n/2
\rfloor -1) < \lfloor n/2 \rfloor$ vertices of $X$ and capture the robber.
\end{proof}

It may be that the graphs $K_4 - 2e$, $K_n$, and $K_m - e$, for $m$ even, are the only connected graphs $G$ satisfying $c_H(G)=\lceil \frac{|V(G)|}{2} \rceil$. The following theorem provides some
evidence for this.

\begin{theorem}\label{thm:kn-2e}
For $n\geq 5$, $c_H(K_n-2e)\leq \lceil \frac{n}{2} \rceil-1$, where $K_n-2e$ is $K_n$ with any two edges removed.
\end{theorem}

\begin{proof}
Note that we have at least two cops for any of the graphs we are considering. If the two edges removed are incident to a common vertex $u$ and the edges removed are $uv$ and $uw$, then the cops place
one cop on $u$ and the rest on distinct vertices not including $u$, $v$, and $w$. The cop on $u$ can see the robber if they occupy $v$ or $w$ and there is at least one cop adjacent to $v$ and $w$.

Therefore, the robber is restricted to playing on $n-(\lceil \frac{n}{2} \rceil-1)-2= \lfloor \frac{n}{2} \rfloor -1$ vertices. Each of the $\lceil \frac{n}{2} \rceil-1$ cops can move to a distinct
vertex that the robber could occupy. Therefore the cops can guarantee a capture.

If the two edges removed are not incident to a common vertex and the edges removed are $ab$ and $cd$, then place a cop on $a$ and $c$. The remaining cops are placed on distinct vertices not including
$a$, $b$, $c$ or $d$. The cop on $a$ can the see the robber if they occupy $b$ and similarly, the cop on $c$ can see the robber if they occupy $d$. The cop on $a$ is adjacent to $d$ and the cop on
$c$ is adjacent to $b$. Therefore, the robber is restricted to playing on $n-1-\lceil \frac{n}{2} \rceil=  \lfloor \frac{n}{2} \rfloor -1$ vertices.  All of the cops can move to distinct vertices
that the robber could occupy and the proof of the upper bound now follows from the previous case.
\end{proof}

\section{Diameter}

\subsection{Diameter at least 3}

In large diameter graphs where the robber may be quite far from the cops, it seems intuitive that the robber is visible more often. This intuition guides the bounds given in this section.

\begin{theorem}\label{td}
If $\mathrm{diam}(G)\ge 3$, then $c_H(G)\le c(G) +2$.
\end{theorem}

\begin{proof} Consider two vertices $u$ and $v$ which are distance at least three apart in $G$. By placing a cop at
both $u$ and $v$, the cops can see every vertex since the robber can never move to a common neighbor of $u$ and $v.$  Therefore, $c(G)$ additional cops can capture the robber on $G$. \end{proof}

Hence, one question is to characterize $G$ with diameter 3 such that $c_H(G) = c(G) + i,$ where $i=0,1,2.$ At present, we do not know of an example of a graph with diameter at least 3 such that
$c_H(G) = c(G) +2$.

We use the notation $\delta(G)$ for the minimum degree of $G$.

\begin{theorem}\label{ttdelta}
If $\mathrm{diam}(G)\geq 3$ and $\delta(G)\leq c(G)$, then $c_H(G)\leq c(G)+1$.
\end{theorem}

\begin{proof}
Place cop $C'$ on a vertex $v$ of minimum degree and $c(G)$-many cops on $N(v)$ in such a way that each vertex in $N(v)$ contains at least one cop. The robber will be visible on their first move as
they will not occupy a vertex of $N[v]$.

The cop $C'$ remains on $v$ and the remaining $c(G)$ cops follow a winning strategy for Cops and Robbers on $G$. Note that this implies that if the robber never enters $N(v)$, then they are captured
eventually.

If the robber enters $N(v)$ and is visible, then $C'$ captures the robber. If the robber enters $N(v)$ and is not visible, then this implies that all $c(G)+1$-many cops are adjacent to the robber.
Any vertex in $N(v)$ that is not adjacent to any cop is visible. Hence, there are at most $\delta(G)$ vertices that the robber could occupy which are adjacent to all the cops; let $Y$ denote this set
of at most $\delta(G)$ vertices. The cops then move to $Y$ to capture the robber. \end{proof}

\begin{corollary}
If $\mathrm{diam}(G)\geq 3$ and $c_H(G) = c(G) +2,$ then $G$ contains a triangle or four-cycle.
\end{corollary}

\begin{proof} By a well-known result of \cite{af}, if $G$ has girth at least 5, then $c(G) \ge \delta(G)$. The result now follows from Theorem~\ref{ttdelta}. \end{proof}

\begin{theorem}
If $G$ has diameter $3$, then for some sufficiently large $N$, whenever $n=|V(G)| \ge N$ we have that $c_H(G) < \lceil n/2 \rceil$. If $G$ is also bipartite, then $N=23.$
\end{theorem}

\begin{proof} Note that by Theorem~\ref{td}, $c_H(G) \le c(G) + 2$. Further, we know from results of Lu and Peng~\cite{lu} that $c(G) = o(n)$, for all connected $G$. Hence, we can choose $N$
 sufficiently large so that $c(G) + 2 < \lceil n/2 \rceil$  and the proof follows.

If $G$ is diameter 3 and bipartite, then, by \cite{lu}, for all $G$, $c(G) \le 2 \sqrt n.$  By Theorem~\ref{ttf}, $c_H (G) \le c(G)+1.$  As $2 \sqrt n + 1 \le \lceil n/2 \rceil -1$, for all $n \ge
23$, the proof follows. \end{proof}

We finish the section with examples of diameter 3 graphs where $c = c_H.$ A \emph{projective plane}\index{projective plane} consists of a set of points and lines satisfying the following axioms:
\begin{enumerate}
\item There is exactly one line incident with every pair of distinct points;

\item There is exactly one point incident with every pair of distinct lines;

\item There are four points such that no line is incident with more than two of them.
\end{enumerate}
Finite projective planes possess $q^{2}+q+1$ points for some integer $q>0$ (called the \emph{order} of the plane). Projective planes of order $q$ exist for all prime powers $q$, and an unsettled
conjecture claims that $q$ must be a prime power for such planes to exist. Given a projective plane $S$, define its \emph{incidence graph} $G(S)$ to be the bipartite graph whose vertices consist of
the points (one partite set), and lines (the second partite set), with a point adjacent to a line if two are incident in $S.$ Note that $G(P)$ is girth 6 and diameter 3; it is known \cite{bonato}
that $c(G(P))=q+1$.

\begin{theorem}
If $G$ is the incidence graph of a projective plane of order $q$, where $q$ is a prime power, then $c_H(G)=c(G)=q+1$.
\end{theorem}

\begin{proof}
The lower bound is due to the fact that $c_H(G)\geq c(G)=q+1$. For the upper bound, the $q+1$ cops begin by occupying $q+1$ distinct vertices in the vertices representing the points of the incidence
graph. If the robber begins by occupying a vertex in the vertices representing the lines of the incidence graph, then if they are not visible, there is only one unique vertex that they can occupy
(since each vertex is degree $q+1$). In that case, the cops win next round since they are all adjacent to them and know their position. If the robber occupies a vertex in the lines of the incidence
graph and is visible, then they must be non-adjacent to each cop or lose in the next round. The cops can move to $q+1$ distinct lines that are adjacent to the $q+1$ neighbors $S$ of the robber, since
any two points in the incidence graph are incident to exactly one line. The robber can only move to one of those points or pass. The cops know if the robber moves to a point, since they become
invisible if they move. The cops then move to $S$ and capture the robber if they moved, or capture them in the next round if the robber did not move.

If the robber begins by occupying a vertex $p$ in the points of the incidence graph, then they are visible and one cop $C_1$ can be moved to force the robber to occupy a line of the incidence graph.
Once the robber occupies a line of the incidence graph, they are visible since the one cop that was moved also occupies a line. The cop $C_1$ moves to $p$ and the remaining $q$ cops move to the $q$
lines that are adjacent to the remaining $q$ neighbors of the robber. The robber will be visible in the next round no matter their move, since there is at least one cop on each of the vertices
representing points and on the vertices representing lines. Further, in the next round, the robber will be adjacent to at least one cop and, thus, they are captured next round.
\end{proof}

\subsection{Diameter 2}

We know by Theorem~\ref{tke} that the hyperopic cop number of a diameter 2 graph can be unbounded as a function of either the cop number or the order of the graph. We present some results for the
diameter 2 case in the present section.

Define $G \vee J$ to be the \emph{join} of $G$ and $J$; note that $G\vee J$ is always diameter at most 2.

\begin{theorem}\label{join}
For a graph $G$ and a graph $J$ with an isolated vertex, we have that $c_H(G \vee J)\leq c_H(G)+1$.
\end{theorem}

\begin{proof}
One cop occupies an isolated vertex in $J$ and the other $c_H(G)$ cops occupy vertices in $G$. If the robber is ever in $J$, then they become visible, and are captured immediately since the
$c_H(G)$-many cops in $G$ are adjacent to the robber. If the robber remains in $G$, then the $c_H(G)$ cops in $G$ will capture the robber.
\end{proof}

We use the notation $\Delta(G)$ for the maximum degree of $G.$

\begin{theorem}\label{tdelta}
If $G$ is diameter 2, then $c_H(G)\le \min\{\delta (G) +1, \Delta(G)\}$.
\end{theorem}

\begin{proof} We prove first that $c_H(G)\le \delta (G) +1$. The cops start on a vertex of minimum degree, say $u$, and all its neighbors.  As the $\mathrm{diam}(G) =2$,
this is a dominating set and the cop on $u$ is at distance two from any possible robber position, the cops can see the robber and therefore, the robber is caught in the next round.

For the upper bound of $\Delta(G),$ the cops start on any vertex of maximum degree, say  $u$, and occupy every vertex in $N(u)$ except for one, say $v$.  The robber begins at a vertex $x$. If $x\ne
v$, then the cop at $u$ is distance two from the robber and so the cops can see the robber. If $x=v$, then the robber is invisible but the cops know they are on $v$, and can catch the robber on the
next round. We conclude that $x$ is adjacent to $v$ but no other vertex in $N(v)$. The cop on $u$ moves to $v$.  If the robber remains on $x$, the cops can catch the robber on the next round.
Therefore, the robber moves to some vertex $y\ne x$.  As $y$ is adjacent to $x$ and at most $\Delta(G)-1$ other vertices, there is a cop which is not adjacent to the robber and therefore, the robber
can be seen by the cops. Further, the cops are positioned on $N(u)$ which is a dominating set as $\mathrm{diam}(G)=2$.  Therefore, the robber can be captured on the next move. \end{proof}

\begin{corollary}
If $\mathrm{diam}(G)=2$ and $G$ has girth 5 and is $r$-regular, then $c_H(G)= r$.
\end{corollary}
\begin{proof}
This follows from Theorem~\ref{tdelta} and by the result of \cite{af}, which states if $G$ is $r$-regular with girth at least 5, $c(G)\ge r$.
\end{proof}

As an application of the corollary, note that the Petersen graph has hyperopic cop number equaling 3.

\section{Planar graphs} \label{secplan}

In this section, we analyze the game of Hyperopic Cops and Robbers played on planar graphs. For planar graphs and outerplanar graphs, we find that the upper bound of the hyperopic cop number matches
the upper bound of the cop number. For a given strategy $\mathcal{S}$ of some set of cops, we say $\mathcal{S}$ is \emph{lonely} if no two of the cops ever occupy the same vertex at the end of a
round throughout the execution of $\mathcal{S}.$ By default, all strategies of a single cop are lonely.

\begin{lemma}\label{lonely}
Suppose that $G$ is a graph and $d$ is a positive integer satisfying the following properties:
\begin{enumerate}
\item The graph $G$ is $K_{d,d+1}$-free.
\item A set of $d
$ cops have a lonely winning strategy $\mathcal{S}$ in the game of Cops and Robbers played on $G$.
\end{enumerate}
Then we have that $c_H(G)\le d.$
\end{lemma}
\begin{proof}  We execute $\mathcal{S}$ while playing Hyperopic Cops and Robbers. If the robber is visible throughout the game, then the cops win as in the classical game of Cops and Robbers.
Therefore, suppose the robber is invisible at some point in the game; in particular, the robber is adjacent to each cop. Let $Y$ be the possible locations of the robber and $X$ the positions of the
cops. Then the subgraph induced by $X\cup Y$ contains a complete bipartite graph $K_{|X|,|Y|}=K_{d,|Y|}$ as the cops occupy distinct vertices. If $|Y|\le d$, then the cops may deduce the robber's
location and move to $Y$ for capture. If $|Y|> d$, then we find a copy of $K_{d,d+1}$, which is a contradiction.
\end{proof}

\begin{lemma}\label{outerplanarcops}
For any outerplanar graph $G$ of order $n\geq2$, there exists a winning lonely strategy for two cops when playing Cops and Robbers.
\end{lemma}

\begin{proof} Throughout, we refer to the winning strategy $\mathcal{S}$ described in \cite{bonato} for two cops playing Cops and Robbers on $G$. We describe how to transform $\mathcal{S}$ into a
winning lonely strategy. Without loss of generality, we may assume the initial positions of the cops are distinct.

In 2-connected graphs, we refer to $\mathcal{S}$ as the \emph{no-cut-vertex strategy}. If there are cut vertices, then we decompose the graph $G$ into maximal 2-connected components or \emph{blocks}.
The strategy $\mathcal{S}$ utilizes retractions to capture the robber's \emph{shadow} in a given block; that is, the image under the appropriate retraction. For completeness, we define these
retractions (see \cite{bonato} for further details). Suppose that the blocks are $G_i$, for $1\le i \le n.$ We retract $G$ onto $G_i$, for any $i$, by the mapping described as follows. Let $x \in V
(G_i) $ and suppose that $x$ is a cut vertex of $G$. All vertices of $G$ that are disconnected from $G_i$ by the deletion of $x$ are mapped to $x$. Vertices of $G_i$ are mapped to themselves.

In the execution of $\mathcal{S},$ we are always in one of the following two situations:

\begin{enumerate}
  \item The graph is 2-connected and the cops occupy the vertices $a_i$ and $a_j$ with degrees at least 3, where $i \neq j$.
  \item The graph is not 2-connected and the cops capture the robber's shadow in one of the blocks $G_i$ using the no-cut-vertex strategy in $G_i$. Note that each $G_i$ has at least two vertices
      and only one cop is required to capture the robber's shadow.
\end{enumerate}

Case (1) never requires the two cops to occupy the same vertex as the cops remain on distinct vertices. Case (2) consists of multiple executions of the strategy in Case (1) and therefore, does not
require the two cops to occupy the same vertex. When the cops must move through a cut vertex from one block to another, they do so one at a time to avoid occupying the same vertex. \end{proof}

Observe that by Theorem~\ref{ttree}, trees are hyperopic cop-win. As the next result shows, all other outerplanar graphs have hyperopic cop number 2.

\begin{theorem}\label{hyperopicouterplanar}
For an outerplanar graph $G$, we have that $c_H(G)\leq 2$.
\end{theorem}
\begin{proof} We employ the winning lonely strategy $\mathcal{S}$ with two cops as described in the proof of Lemma~\ref{outerplanarcops}. The result now follows by Lemma~\ref{lonely} since outerplanar graphs
are $K_{2,3}$-free.
\end{proof}

We proceed in an analogous fashion for planar graphs.

\begin{lemma}\label{planarcops}
For any planar graph $G$ of order $n\geq3$, there exists a winning lonely strategy for three cops when playing Cops and Robbers.
\end{lemma}

\begin{proof} Throughout, we refer to the winning strategy $\mathcal{S}$ described in \cite{bonato} for three cops playing Cops and Robbers on $G$. We describe how to transform $\mathcal{S}$ into a
winning lonely strategy. Without loss of generality, we may assume the initial positions of the cops are distinct.

In the execution of $\mathcal{S},$ we are always in one of the following two situations:

\begin{enumerate}
  \item Two cops $C_1$ and $C_2$ are guarding disjoint shortest paths whose union forms a cycle $X$. The third cop $C_3$ is moving either towards or onto a shortest path $P$ so as to eventually
      guard it and thus, release one of $C_1$ or $C_2$ from roles of guarding a path.
  \item There are at least two \emph{free} cops $C_2$ and $C_3$; that is, there are two cops who are not guarding shortest paths.
\end{enumerate}

We first consider Case (1). It is evident that the cops $C_1$ and $C_2$ never occupy the same vertex. The path $P$ must be internally disjoint from $X$ (as the path terminates on $X$). Therefore,
once $C_3$ occupies a vertex of $P$, none of the cops will ever occupy the same vertex since the three cops can guard the paths as three internally disjoint paths.

Suppose that $C_3$ is moving towards $P$ and wants to occupy the same vertex, say $u$, as one of $C_1$ or $C_2$. If one of them wants to pass (say $C_1$), then $C_3$ and $C_1$ change roles with $C_3$
moving to $u$ and $C_1$ moving towards $P$. If $C_1$ and $C_3$ both want to move to a given, adjacent vertex $x$, then $C_1$ moves to $x$ and $C_3$ passes. In the following round, $C_3$ moves to $x$
unless $C_1$ passes, in which case we follow the procedure above ($C_3$ and $C_1$ change roles with $C_3$ moving to $x$ and $C_1$ moving towards $P$). In this way, $C_3$ eventually gets to $P$, while
$X$ remains guarded and no two cops occupy the same vertex.

Case (2) is analogous, with the free cops $C_2$ and $C_3$ moving accordingly with $C_1$ as needed (in a similar fashion as above). Note that at most one of $C_2$ and $C_3$ will be moving at
a time towards a shortest path to guard it.
\end{proof}

\begin{theorem}\label{hyperopicplanar}
For a planar graph $G$, we have that $c_H(G)\leq 3$.
\end{theorem}

\begin{proof} We employ the winning lonely strategy $\mathcal{S}$ with three cops as described in the proof of Lemma~\ref{planarcops}.
The result now follows by Lemma~\ref{lonely} since planar graphs
are $K_{3,4}$-free.
\end{proof}

\section{Hyperopic density}\label{secdense}

For infinite graphs, an established approach to studying various graph parameters is via graph densities. For a non-negative integer-valued graph parameter $f$ and a graph $G$ such that $f(G)\le
|V(G)|$, define $f(G)/|V(G)|$ to be the density of $f$. To analyze the cop number of infinite graphs, we consider the \emph{cop density} of a finite graph first introduced in \cite{BHW}. Define
\begin{equation*}
D_{c}(G)=\frac{c(G)}{|V(G)|}.
\end{equation*}%
Note that $D_{c}(G)$ is a rational number in $(0,1]$. We extend the definition of $D_{c}$ to infinite graphs by considering limits of chains of finite graphs. In this way, the cop density for
infinite graphs is a real number in $[0,1]$.

A \emph{chain} of induced subgraphs in $G$, denoted $\mathcal{C}=(G_{n}:n\in \mathbb{N })$, has the property that $G_n$ is an induced subgraph of $G_{n+1}$ for all $n.$ The \emph{limit} of the chain
$\mathcal{C}$ is defined as the subgraph $G'$ such that
\begin{equation*}
V(G')=\bigcup\limits_{n\in \mathbb{N}}V(G_{n}), \qquad
E(G')=\bigcup\limits_{n\in \mathbb{N}}E(G_{n}).
\end{equation*}
We write $G'=\lim_{n\rightarrow \infty }G_{n}.$ We say that the chain $\mathcal{C}$ is \emph{full} if $G' = G.$  Note that every countable graph $G$ is the limit of a full chain of finite induced
graphs, and there are infinitely many distinct full chains
with limit $G$.

Suppose that $G=\lim_{n\rightarrow \infty }G_{n}$, where $\mathcal{C}=(G_{n}:n\in \mathbb{N%
})$ is a fixed full chain of induced subgraphs of $G$. Define
\begin{equation*}
D(G,\mathcal{C})=\lim_{n\rightarrow \infty }D_{c}(G_{n}),
\end{equation*}%
if the limit exists (and then it is a real number in $[0,1]$). This is the
\emph{cop density} \emph{of }$G$\emph{\ relative to} $\mathcal{C}$; if $%
\mathcal{C}$ is clear from context, we refer to this as the \emph{cop density} \emph{of} $G$. We will only consider graphs and chains where this limit exists. We may define the \emph{hyperopic cop
density}, written $D_H(G,\mathcal{C})$, in an analogous fashion.

Under fairly weak assumptions, it was shown that the cop density of a countable graph can be any real number in $[0,1]$ depending on the chain used; see \cite{BHW}. If we insist, however, that all the
elements of the chain $\mathcal{C}$ are connected, then the situation for cop density changes radically. By Frankl's bound on the cop number of connected graphs $G$, $c(G)=o(|V(G)|)$ (see
\cite{frankl}), it follows that
\begin{equation*}
D(G,\mathcal{C})=0.
\end{equation*}

We show that the situation for densities of the hyperopic cop number is very different.

\begin{theorem}\label{mainden}  There exists a countable graph $G$ such that for all $r\in
\lbrack 0,1/2]$, there is a chain $\mathcal{C}$ with limit $G$ such that each element of $\mathcal{C}$ is connected and $D_H(G,\mathcal{C})=r$.
\end{theorem}

We first need Lemma~\ref{vlem} (and its proof), which will then be followed by the proof of the main Theorem.

\begin{lemma} \label{vlem}
For all $r,s \ge 2,$ $c_H(K_r \vee \overline{K_s}) = \lfloor \frac{r}{2}\rfloor +1 .$
\end{lemma}

Vertices of the clique $K_r$ are called \emph{clique vertices} and vertices of the co-clique $\overline{K_s}$ are called \emph{co-clique vertices}. (Note that each vertex in $K_r \vee \overline{K_s}$
is either a clique or co-clique vertex.) Lemma~\ref{vlem} shows that however many co-clique vertices we add, the cop number increases only by 1. Further, for every two clique vertices we add, we need
an additional cop.

\begin{proof}
For the upper bound in the case $r$ is even, we place $\frac{r}{2}$ cops in $K_r$ and one cop in $\overline{K_s}.$ The robber cannot choose a co-clique vertex as they would be visible and captured in
the next round. Hence, the robber must begin on a clique vertex and is captured by Theorem~\ref{tmax}. An analogous argument works when $r$ is odd, with the exception that if the robber is on a
clique vertex, the additional cop in $\overline{K_s}$ moves to $K_r$ to capture with the $\lfloor \frac{r}{2} \rfloor$ other cops.

The lower bound follows since the robber is always invisible as long as they stay on clique vertices. To see this, suppose we play with $k$ cops, where $k <\lfloor \frac{r}{2}\rfloor +1 .$ First,
consider $r$ even. The case $r$ is odd is analogous and so is omitted. If there are fewer than $\frac{r}{2}$ cops in $K_r$, then the robber stays invisible by starting on a clique vertex and no
matter the position of the cops, they cannot guarantee the location of the robber.
\end{proof}

\begin{proof} Let $(p_{n}:n\in \mathbb{N})$ be a sequence of rationals in $(0,1/2]$  such
that $\lim_{n\rightarrow \infty }p_{n}=r$, with $p_{0}=1/2$ such that $p_n<\frac 12$ for $n>0$. Let $G = K_{\aleph_0} \vee \overline{K_{\aleph_0} },$ where $K_{\aleph_0}$ is the countably infinite
clique.

We construct a chain $\mathcal{C}=(G_{n}:n\in \mathbb{N})$ in $G$ such that $G=\lim_{n\rightarrow \infty }G_{n}$, and with the property that $D_{H}(G_{n})=p_{n}$. Enumerate $V(G)$ as $\{x_{n}:n\in
\mathbb{N}\}$. We proceed inductively on $n$. For $n=0$, let $G_{0}$ be the subgraph induced by $x_{0}$ and three additional vertices so that $G_{0}$ has exactly two clique vertices and two co-clique
vertices. Hence, $G_0$ is isomorphic to $K_4$ minus an edge. Then we have that
\begin{equation*}
\frac{c_H(G_{0})}{|V(G_{0})|}=\frac 12 =p_{0}.
\end{equation*}

Fix $n\geq 1$, suppose the induction hypothesis holds for all $k\leq n$, and let $p_{n+1}=\frac{a}{b}$, where $a$ and $b$ are positive integers. Further suppose for an inductive hypothesis that
$\{x_{0},\ldots ,x_{n}\}\subseteq V(G_{n})$. Without loss of generality, as $r\in \lbrack 0,1/2]$ we may assume $0<2a<b$ and $\gcd (a,b)=1$.

We add vertices to $G_{n}$ in stages. Define $G_{n+1}^{\prime }$ to be the graph induced by $V(G_{n})\cup \{x_{n+1}\}$. Suppose that $c_H(G_{n+1}^{\prime
})=a^{\prime }$ and $|V(G_{n+1}^{\prime })|=b^{\prime }$. If $\frac{%
a^{\prime }}{b^{\prime }}=\frac{a}{b}$, then let $G_{n+1}=G_{n+1}^{\prime }$%
. Otherwise, we add some new vertices to adjust the density $D_{H}(G_{n+1}^{\prime })$. Note that $a' = \lfloor \frac{i}{2} \rfloor +1$ and $b'=i+j,$ for some positive integers $i$ and $j.$

We may assume that $\frac{a^{\prime }}{b^{\prime }}<\frac{a}{b}$ by adding an appropriate number of co-clique vertices. In this way, $b^{\prime }$ will become larger, while $a^{\prime }$ will not
change.

We add $x$ clique vertices and $y$ co-clique vertices to $G_{n+1}^{\prime }$ to form the induced subgraph $G_{n+1}$ so that
\begin{equation*}
D_{H}(G_{n+1})=\frac{c_H(G_{n+1})}{|V(G_{n+1})|}=\frac{a}{b}.
\end{equation*}
This is possible if we can solve the equation
\begin{equation*}
\frac{a}{b}=\frac{\lfloor (i+x)/2 \rfloor+1}{i+j+x+y}.
\end{equation*}
We find a solution where $x$ is even.  In this case this equation is equivalent to
\begin{equation}
x b/2 -ax -ay = ai + aj - \lfloor i/2 \rfloor b - b = \beta .  \label{la1}
\end{equation}
Note that $\beta >0$; otherwise, we may choose $i$ and $j$ large so this occurs.

There are two cases to consider. First, suppose $b$ is even. From (\ref{la1}), we obtain a linear Diophantine equation
\begin{equation*}
(b/2 - a)x-ay=\beta,
\end{equation*}
where $b/2-a > 0$ (recall that $2a<b$). As $\gcd (b/2-a,-a)=\gcd (a,b)=1$, (\ref{la1}) has infinitely many solutions. The general integer solution of (\ref{la1}) is
\begin{equation}
x=x_{0}-at,y=y_{0}-(b/2-a)t,  \label{ma}
\end{equation}%
where $(x_{0},y_{0})$ is a particular fixed solution, and $t$ is an integer.  Note that the coefficients of $t$ in (\ref{ma}) are both negative, so we may choose an appropriate $t<0$ to ensure an
integer solution of (\ref{la1}) $(x,y)$ with $x,y\geq 0$.  Further, since $\gcd(a,b)=1$ and $b$ is even, it follows that $a$ is odd.  Hence, $t$ may be chosen so that $x$ is even.

Now suppose $b$ is odd. From (\ref{la1}) we obtain a linear Diophantine equation
\begin{equation}
(b - 2a)x-2ay=2\beta, \label{mb}
\end{equation}
where $b/2-a > 0$.  As $\gcd (b-2a,-2a)=\gcd (a,b)=1$, (\ref{mb}) has infinitely many solutions. Further, rearranging (\ref{mb}) we obtain that $(b - 2a)x=2\beta+2ay$.  Hence, as $b-2a$ is odd,  it
follows that in any solution of (\ref{mb}), $x$ must be even. The general integer solution of (\ref{la1}) is
\begin{equation}
x=x_{0}-2at,y=y_{0}-(b-2a)t,  \label{mc}
\end{equation}%
where $(x_{0},y_{0})$ is a particular fixed solution, and $t$ is an integer. The coefficients of $t$ in (\ref{mc}) are both negative, so we may choose an appropriate $t<0$ to ensure an integer
solution of (\ref{la1}) $(x,y)$ with $x,y\geq 0$.  This completes the induction step in constructing $G_{n+1}$.

As $\{x_{0},\ldots ,x_{n}\}\subseteq V(G_{n})$ for all $n\in \mathbb{N}$, we have that $\mathcal{C=}(G_{n}:n\in \mathbb{N})$ is a full chain for $G$. Further,
\begin{equation*}
D_H(G,\mathcal{C})\mathcal{=}\lim_{n\rightarrow \infty }p_{n}=r. \qedhere
\end{equation*}

\end{proof}


\begin{thebibliography}{99}

\bibitem{af} M.\ Aigner, M.\ Fromme, A game of cops and robbers, \emph{Discrete Applied Mathematics} \textbf{8} (1984) 1--12.

\bibitem{bonato1} W.\ Baird, A.\ Bonato, Meyniel's conjecture on the cop number: a survey,  \emph{Journal of Combinatorics} \textbf{3} (2012) 225--238.

\bibitem{market} V.\ Boginski, S.\ Butenko, P.M.\ Pardalos, On structural properties of the market graph, In: \emph{Innovation in Financial and Economic Networks}, Edward Elgar
    Publishers, pp.\ 29--45, 2003.

\bibitem{BHW} A.\ Bonato, G.\ Hahn, C.\ Wang,  The cop density of a graph, \emph{Contributions to Discrete Mathematics} \textbf{2} (2007) 133--144.

\bibitem{ILAT} A.\ Bonato, E.\ Infeld, H.\ Pokhrel, P.\ Pra\l{}at, Common adversaries form alliances: modelling complex networks via anti-transitivity, In: \emph{Proceedings of the 14th Workshop on
    Algorithms and Models for the Web Graph}, 2017.

\bibitem{bonato} A.\ Bonato, R.J.\ Nowakowski, \emph{The Game of Cops and Robbers on Graphs}, American Mathematical Society, Providence, Rhode Island, 2011.

\bibitem{by} A.\ Bonato, B.\ Yang, Graph searching and related problems, invited book chapter in: \emph{Handbook of Combinatorial Optimization}, editors P. Pardalos, D.Z. Du, R. Graham, 2011.

\bibitem{vis} N.E.\ Clarke, D.\ Cox, C.\ Duffy, D.\ Dyer, S.\ Fitzpatrick, M.E.\ Messinger, Limited visibility Cops and Robbers, Preprint 2017.

\bibitem{fomin} F.V.\ Fomin, D.M.\ Thilikos, An annotated bibliography on guaranteed graph searching, \emph{Theoretical Computer Science} \textbf{399} (2008) 236--245.

\bibitem{frankl} P.\ Frankl, Cops and robbers in graphs with large girth and Cayley graphs, \emph{Discrete Applied Mathematics} \textbf{17} (1987) 301--305.

\bibitem{guo} W.\ Guo, X.\ Lu, G.M.\ Donate, S.\ Johnson, The spatial ecology of war and peace, Preprint 2017.

\bibitem{lu} L.\ Lu, X.\ Peng, On Meyniel's conjecture of the cop number, \emph{Journal of Graph Theory}  \textbf{71} (2012) 192--205.

\bibitem{nw} R.J.\ Nowakowski, P.\ Winkler, Vertex-to-vertex pursuit in a graph, \emph{Discrete Mathematics} \textbf{43} (1983) 235--239.

\bibitem{q} A.\ Quilliot, Jeux et pointes fixes sur les graphes, \emph{Th\`{e}se de 3\`{e}me cycle}, Universit\'{e} de Paris VI, 1978, 131--145.

\bibitem{q2} A.\ Quilliot, \emph{Probl\`{e}mes de jeux, de point Fixe, de connectivit\'{e} et de repres\'{e}sentation sur des graphes, des ensembles ordonn\'{e}s et des hypergraphes}, Th\`{e}se
    d'Etat, Universit\'{e} de Paris VI, 1983, 131--145.

\bibitem{s} B.S.W.\ Schroeder, The copnumber of a graph is bounded by $\lfloor \frac{3}{2} genus(G)\rfloor +3$, \emph{Categorical Perspectives}, Trends Math., Birkh\"{a}user, Boston, MA, 2001, 243--263.

\bibitem{west} D.B.\ West, \emph{Introduction to Graph Theory, 2nd edition}, Prentice Hall, 2001.

\end{thebibliography}
\end{document}